\documentclass{gtart_a}
\pdfoutput=1

\usepackage[all]{xy}
\SelectTips{cm}{}
\usepackage{graphicx}

\title[A dilogarithmic formula for the Cheeger--Chern--Simons class]{A dilogarithmic formula\\for the Cheeger--Chern--Simons class}

\author{Johan L\,Dupont}
\givenname{Johan}
\surname{Dupont}
\address{Department of Mathematics\\
University of Aarhus\\\newline
DK-8000 \AA rhus\\Denmark}
\email{dupont@imf.au.dk}
\urladdr{}

\author{Christian K\,Zickert}
\givenname{Christian}
\surname{Zickert}
\address{Department of Mathematics\\
Columbia University\\\newline
New York, NY 10027\\USA}
\email{zickert@math.columbia.edu}
\urladdr{}

\volumenumber{10}
\issuenumber{}
\publicationyear{2006}
\papernumber{33}
\lognumber{0638}
\startpage{1347}
\endpage{1372}

\doi{}
\MR{}
\Zbl{}

\arxivreference{}  %%% not in arXiv


\received{2 August 2005}
\revised{2 August 2006}
\accepted{14 June 2006}
\published{26 September 2006}
\publishedonline{26 September 2006}
\proposed{Shigeyuki Morita}
\seconded{Joan Birman, Robion Kirby}
\corresponding{}
\editor{Walter Neumann, Robion Kirby}
\version{}

\let\xysavmatrix\xymatrix
\def\xymatrix{\disablesubscriptcorrection\xysavmatrix}
\AtBeginDocument{\let\bar\wbar\let\hat\what}
\AtBeginDocument{\renewcommand{\R}{\mathbb R}} 
\AtBeginDocument{\renewcommand{\C}{\mathbb C}}
\AtBeginDocument{\renewcommand{\Z}{\mathbb Z}}
\AtBeginDocument{\renewcommand{\Q}{\mathbb Q}}
\AtBeginDocument{\renewcommand{\epsilon}{\varepsilon}}
\def\hatcB{\widehat{\mathcal B\mskip1mu}{}}
\allowdisplaybreaks

\makeatletter
\def\cnewtheorem#1[#2]#3{\newtheorem{#1}{#3}[section]
\expandafter\let\csname c@#1\endcsname\c@thm}

\theoremstyle{plain}
\newtheorem{thm}{Theorem}[section]
\cnewtheorem{lemma}[thm]{Lemma}
\cnewtheorem{prop}[thm]{Proposition}
\cnewtheorem{cor}[thm]{Corollary}

\theoremstyle{definition}
\cnewtheorem{defn}[thm]{Definition}

\theoremstyle{remark}
\newtheorem*{remark*}{Remark}
\cnewtheorem{remark}[thm]{Remark}
\makeatother  %  move after \newtheorem block

\makeautorefname{defn}{Definition}
\DeclareMathOperator{\Log}{Log}
\DeclareMathOperator{\Li}{Li_2}
\DeclareMathOperator{\FT}{FT}
\DeclareMathOperator{\Arg}{Arg}
\DeclareMathOperator{\Vol}{Vol}
\DeclareMathOperator{\Imag}{Im}
\DeclareMathOperator{\SL}{SL}
\DeclareMathOperator{\PSL}{PSL}
\DeclareMathOperator{\Hom}{Hom}
\DeclareMathOperator{\cut}{cut}

\newcommand{\F}{\mathbb F}
\newcommand{\Cremove}{\C\backslash\{0,1\}}

\newcommand{\B}{\mathcal B}
\newcommand{\Pre}{\mathcal P}
\numberwithin{equation}{section}
\def\cxymatrix#1{\xy*[c]\xybox{\xymatrix#1}\endxy}

\begin{document}

\begin{asciiabstract}
We present a simplification of Neumann's formula for the universal
Cheeger-Chern-Simons class of the second Chern polynomial. Our
approach is completely algebraic, and the final formula can be applied
directly on a homology class in the bar complex.
\end{asciiabstract}

\begin{htmlabstract}
We present a simplification of Neumann's formula for the universal
Cheeger&ndash;Chern&ndash;Simons class of the second Chern polynomial. Our
approach is completely algebraic, and the final formula can be
applied directly on a homology class in the bar complex.
\end{htmlabstract}

\begin{abstract} We present a simplification of Neumann's formula for the universal
  Cheeger--Chern--Simons class of the second Chern polynomial. Our
  approach is completely algebraic, and the final formula can be
  applied directly on a homology class in the bar complex.
\end{abstract}

\maketitle

\section*{Introduction}
In the famous papers
\cite{CheegerSimons} and \cite{ChernSimons}, J\,Cheeger, S\,Chern and
J\,Simons define characteristic classes for flat $G$--bundles. Each
such characteristic class is given by a corresponding universal
cohomology class in $H^*(BG^\delta,\C/\Z)$, where $\delta$ denotes
discrete topology.  The cohomology of the classifying space of a
discrete group is isomorphic to the Eilenberg--Maclane group
cohomo\-logy, and it has been a long standing problem to find explicit
formulas for the universal classes directly in terms of the bar
complex.  In \cite{Dupont}, the first author proved that the universal
Cheeger--Chern--Simons (C--C--S) class for the group $\SL(2,\C)$
associated to the second Chern polynomial is given up to a $\Q/\Z$ indeterminacy by a dilogarithmic
formula defined on the Bloch group $\B(\C)$. 

An element of $\B(\C)$ is a formal sum of
cross-ratios (see below), but the cross-ratio alone
does not seem to carry enough information to get rid of the $\Q/\Z$
indeterminacy on the real part. Neumann \cite{Neumann} constructs 
an extended Bloch group $\hatcB(\C)$, where elements, in addition to the
cross-ratio, also contain information of two choices of
logarithms. It follows from Neumann's article that this additional
information is exactly what is needed to remove the  $\Q/\Z$
indeterminacy. 
He shows that there is an isomorphism \[\lambda\co
H_3(\PSL(2,\C))\cong \hatcB(\C),\] and furthermore that there is a natural extension of the dilogarithmic formula
from \cite{Dupont} to $\hatcB(\C)$ such that the composition of
$\lambda$ with the dilogarithm is exactly the universal C--C--S class. 
The formula also gives a combinatorial formula for the volume and the Chern--Simons invariant of a complete hyperbolic manifold $M$ with finite volume, since the C--C--S class evaluated on the canonical flat $\PSL(2,\C)$--bundle over $M$ equals $i(\Vol + i\text{CS})$.  This is shown by Neumann and Yang \cite{NeumannYang}.

The isomorphism $\lambda$ is defined
by representing an element of $H_3(\PSL(2,\C))$ by a ``quasisimplicial
complex,'' and the appropriate choices of logarithms required to obtain
an element in $\widehat B(\C)$ are found by studying combinatorial properties of this complex.
We construct a map similar to Neumann's using $\SL(2,\C)$ instead
of $\PSL(2,\C)$ and the more extended Bloch group from \cite[Section 8]{Neumann} instead of the extended Bloch group.
We thus answer
affirmatively a question raised about the
relation of this group to $H_3(\SL(2,\C))$ \cite[page 443]{Neumann}.
The definition of our map uses only simple
homological algebra, and we obtain a formula which enables
us to calculate the universal \mbox{C--C--S} class directly from a
representative of a homology class in the bar complex. 
All geometry is replaced by algebra which vastly
simplifies the proofs.

We give a brief overview of the contents: 
In \fullref{Overview} we
review the basic theory of the C--C--S classes,
group homology and the Bloch group. Many details are included in
order to make the paper self-contained. 
In \fullref{extb} we recall Neumann's definition of the (more) extended
Bloch group. This overlaps with Neumann's paper but for the sake of
completeness, we include most of the details.
In \fullref{config}, we construct a map $\widehat \lambda\co
H_3(\SL(2,\C))\to \hatcB(\C)$ by describing a way of detecting the appropriate two
choices of logarithms directly from a tuple of group elements. 
The idea is that the extra information
can be found in $\C^2\backslash\{0\}$ rather than $S^2$ using the Hopf map.
In \fullref{relccc} we show that our map actually calculates the
C--C--S class and show that $\widehat \lambda$ is surjective with kernel
of order $2$. Finally, we show in the appendix that our
definition of the extended Bloch group agrees with that of Neumann.

\begin{remark*} The reader should keep in mind that whenever we mention the extended Bloch group, 
we always mean the \emph{more\/} extended Bloch group from \cite[Section 8]{Neumann}.
Neumann uses the notation $\mathcal E \B(\C)$ for this group but we will use the notation $\widehat \B(\C)$ 
even though this conflicts with the notation in \cite{Neumann}.
\end{remark*}

\subsubsection*{Acknowledgements}
This work was partially supported by The Danish Natural Science
Research Council (Statens Naturvidenskabelige
Forsk\-nings\-r\aa d), Denmark.

\section{Preliminaries}\label{Overview}
In this section we review some basic theory and introduce our
terminology. Throughout, $\F$ always denotes either $\R$ or $\C$.
\subsection{The Cheeger--Chern--Simons classes}
We here recall some facts about the C--C--S classes that we shall need.
For their construction and basic properties, we refer to
\cite{CheegerSimons} or \cite{ChernSimons}.

Let $G$ be a Lie group with finitely many components and let $I^k(G,\F)$ denote the group of invariant
polynomials. Recall from classical Chern--Weil
theory that there is a natural homomorphism
 \begin{equation*}W\co I^k(G,\F)\to H^{2k}(BG,\F).
\end{equation*}
Let $r$ denote the map $H^*(BG,\Z)\to H^*(BG,\F)$ induced by the inclusion.
The C--C--S classes are defined from the following data:
\begin{enumerate}
\item An invariant polynomial $P\in I^k(G,\F)$.
\item A class $u\in H^{2k}(BG,\Z)$ satisfying $W(P)=ru$.
\end{enumerate}
Let
$
  K^k(G,\F)=\big\{(P,u)\in I^k(G,\F)\times H^{2k}(BG,\Z)\mid W(P)=ru\big\}.
$

Let $G^\delta$ denote the underlying discrete group of $G$. In
\cite{CheegerSimons} and \cite{ChernSimons},
the authors describe a way of associating a cohomology
class $\hat P(u)$ in $H^{2k-1}(BG^\delta,\F/\Z)$ to an element
$(P,u)$ in $K^k(G,\F)$. This association is natural in the following sense:
\begin{thm}\label{godformel} Let $\phi\co G\to H$ be a Lie group
  homomorphism between Lie groups with finitely many components. The diagram below is commutative.
 \[\xymatrix{{K^k(H,\F)\ar[r]^{\phi^*}}\ar[d]^{\textnormal{C--C--S}}&{K^k(G,\F)}\ar[d]^{\textnormal{C--C--S}}\\
             {H^{2k-1}(BH^\delta,\F/\Z)}\ar[r]^{\phi^*}&{H^{2k-1}(BG^\delta,\F/\Z)}}\]
\end{thm}
\begin{remark}
In the following we shall only be interested in the C--C--S classes corresponding to the second Chern polynomial and
the first Pontrjagin polynomial. In both cases $u$ is just the corresponding Chern class or Pontrjagin class, 
and we simply denote the associated C--C--S classes $\hat C_2$ and $\hat P_1$. 
\end{remark}
\subsection{The homology of a group}\label{grouphomology}
Let $G$ be a group. For a right $G$--module $A$, we let $A_G$ denote
the group $A\otimes_{\Z[G]}\Z$, where $\Z$ is regarded as a trivial $G$--module.
The homology of $G$ is by definition the homology of the
complex $(P_*)_G$, where $P_*$ is a projective resolution of $\Z$ by
right $G$--modules. The following general construction of
a projective resolution is of parti\-cular interest to us:
For $X$ a set, let $C_*(X)$ be the acyclic complex of free
abelian groups, which in dimension $n$ is generated by $(n+1)$--tuples of
elements in $X$. The differential is given by \[\partial(x_0,\dots,x_n)=\sum_{i=0}^n(-1)^i(x_0,\dots,\hat x_i,\dots,x_n).\]
In particular for $X=G$, the diagonal left $G$--action on tuples makes
$C_*(G)$ into a complex of $G$--modules (considered as right modules in the standard way) and $C_*(G)$ augmented by the map
$C_0(G)\to \Z$ given by $(g_0)\mapsto 1$ 
is a free resolution of $\Z$. The complex $C_*(G)_G$ thus
calculates the homology of $G$.

There is another description of this complex.
Consider the complex $B_*(G)$ of free abelian groups, which in dimension
$n$ is generated by symbols $[g_1\vert\cdots\vert g_n]$ and with
differential given by
\begin{align*}\partial [g_1\vert\cdots\vert g_n]= [g_2\vert\cdots\vert g_n]&+
\sum_{i=1}^{n-1}(-1)^i[g_1\vert\cdots\vert g_ig_{i+1}\vert\cdots\vert g_n]\nonumber\\
&+(-1)^n[g_1\vert\cdots\vert g_{n-1}].\end{align*}
This complex is isomorphic to $C_*(G)_G$ via the map 
\begin{equation}\label{inhom}[g_1\vert\cdots\vert g_n]\mapsto
  (1,g_1,g_1g_2,\dots,g_1g_2\cdots g_n)\end{equation} with inverse
\begin{equation}\label{hom}(g_0,\dots,g_n)\mapsto [g_0^{-1}g_1\vert\cdots\vert g_{n-1}^{-1}g_n].
\end{equation}
Hence, we can represent a homology class in $H_n(G)$ 
either by a chain in $C_n(G)$ or by a cycle in $B_n(G)$. 
These two ways of representing homology classes are called the 
\emph{homogenous\/} and the \emph{inhomogenous\/} representation,
respectively.

%\begin{cor}\label{homo}
%Let $D_*(G)$ be a $G$--subcomplex of $C_*(G)$ satisfying the hypo\-thesis of the lemma.
%Then $D_*(G)\otimes_{\Z[G]} \Z=D_*(G)_G$ calculates the homology of $G$. 
%\end{cor}

Let $M$ be a left $G$--module. The cohomology $H^*(G,M)$ is defined as
the homology of the complex $\Hom_{\Z[G]}(P_*,M)$, where $P_*$, this time, is
a projective resolution of $\Z$ by \emph{left\/} $G$--modules.
Regarding a divisible abelian group $A$ as a trivial $G$--module, we
have by the universal coefficient theorem a natural isomorphism
\begin{equation*}\label{univ} H^n(G,A)=\Hom(H_n(G),A).\end{equation*}
It is well known that the homology of a group is isomorphic to the
singular homology of its classifying space, and since the abelian
group $\F/\Z$ is obviously divisible, we can regard the C--C--S classes
as homomorphisms from $H_3(G)$ to $\F/\Z$. It is an interesting problem
to try to find explicit formulas for the C--C--S classes directly in
terms of the resolution $C_*(G)$ (or some subcomplex). We shall
investigate this in the following sections.

%We will also be interested in certain subcomplexes of $C_*(G)$.
We conclude the section with a little lemma that will be useful later.
For each $g\in G$ there is a map $s_g\co C_*(G)\to C_*(G)$ given by
$s_g(g_0,\dots,g_n)=(g,g_0,\dots ,g_n)$. 
\begin{lemma}\label{subcomp}
Let $D_*(G)$ be a $G$--subcomplex of $C_*(G)$. Suppose that for each cycle $\sigma$ in $D_*(G)$,
there exists a point $g(\sigma)$ in $G$ such that $s_{g(\sigma)}\sigma$ is in $D_{n+1}(G)$.
Then $D_*(G)$ is acyclic and $D_*(G)_G$ calculates the homology of $G$.
\end{lemma}
\begin{proof} 
Note that $\partial s_g(g_0,\dots,g_n)=(g_0,\dots,g_n)-s_g(\partial (g_0,\dots,g_n))$.
Let $\sigma$ be a cycle in $D_*(G)$. Since $\partial\sigma=0$ we have $\sigma = \partial s_{g(\sigma)}\sigma$,
that is, $\sigma$ is a boundary.
\end{proof}

\subsection{The Bloch group}\label{blochgroup}
In all the following, we let $G$ denote the group $\SL(2,\C)$.
\begin{defn}\label{predefn} The \emph{pre-Bloch group\/} $\Pre(\C)$ is an abelian group generated by symbols $[z]$, $z\in \Cremove$ subject to the relation
\begin{equation}\label{fiveterm}
[x]-[y]+\Big[\frac{y}{x}\Big]-\Big[\frac{1-x^{-1}}{1-y^{-1}}\Big]+\Big[\frac{1-x}{1-y}\Big]=0.
\end{equation} 
This relation is called the \emph{five term relation\/}.
\end{defn}
In \cite{Dupont} and \cite{DupontSah} the five term relation is
different, but this is because of the different definition of the
cross-ratio \eqref{cr}.
\begin{defn} The \emph{Bloch group\/} $\B(\C)$ is the kernel of the homomorphism
\[\nu\co \Pre(\C)\to \C^*\wedge\C^*\]
to the second exterior power of the abelian group $\C^*$
defined by mapping a generator $[z]$ to $z\wedge (1-z)$. 
\end{defn}
There is an important interpretation of the pre-Bloch group in terms
of a homology group. Recall the notation from \fullref{grouphomology}. Let $C_*^{\neq}(S^2)$ denote the subcomplex of $C_*(S^2)$ consisting of tuples of distinct elements.
Recall that $G=\SL(2,\C)$ acts on $S^2=\C\cup\{\infty\}$ by M\"obius transformations, that is,
\begin{equation*}\begin{pmatrix}a&b\\c&d\end{pmatrix}z=\frac{az+b}{cz+d}.\end{equation*}
Via this action, the complex $C_*^{\neq}(S^2)$ becomes a complex of $G$--modules.
The action is $3$--transitive and four distinct points $z_0,\dots ,z_3$ are determined up to the action
by the \emph{cross-ratio\/}
\begin{equation}\label{cr}
z=[z_0:z_1:z_2:z_3]:=\frac{(z_0-z_3)(z_1-z_2)}{(z_0-z_2)(z_1-z_3)}.
\end{equation}
Note that in \cite{Dupont} and \cite{DupontSah} the cross-ratio is defined to be the
  reciprocal of \eqref{cr}.
It follows that $C_3^{\neq}(S^2)_G$ is just the free abelian group on $\Cremove$. 
One easily checks that the five term relation is equivalent
to the relation 
\begin{equation*}\sum_{i=0}^4(-1)^i[z_0:\dots:\hat z_i:\dots:z_4]=0.\end{equation*}
This means that the kernel of the cross-ratio map $\sigma\co C_3^{\neq}(S^2)\to \Pre(\C)$ is exactly the boundaries.
Since $C_2^{\smash{\neq}}(S^2)_G=\Z$ by $3$--transitivity, $C_3^{\smash{\neq}}(S^2)_G$ consists entirely of cycles, and $\sigma$ induces an isomorphism
\begin{equation*}\sigma\co  H_3(C_*^{\neq}(S^2)_G)\to \Pre(\C).\end{equation*}
We have the following relations in the pre-Bloch group 
\cite[Lemma 5.11]{DupontSah}:
\begin{equation*}[x]=\Big[\frac{1}{1-x}\Big]=\Big[1-\frac{1}{x}\Big]=-\Big[\frac{1}{x}\Big]=-\Big[\frac{x}{x-1}\Big]=-[1-x]\end{equation*}
If we extend the cross-ratio by setting $[z_0:z_1:z_2:z_3]=0$ if there
are equals among $z_0,\dots,z_3$, it follows from the above relations that
$\sigma$ can be extended to $H_3(C_*(S^2)_G)$.
We omit the details.
We can now define a map \begin{equation*}\lambda\co  H_3(G)\to \Pre(\C)\end{equation*} as the composition
\begin{equation*}\label{lamda}
\xymatrix{{H_3(G)}\ar[r]&{H_3(C_*(S^2)_G)}\ar[r]^-{\sigma}&{\Pre(\C)}}\end{equation*}
where the left map is induced by
\begin{equation*}
C_3(G)\to C_3(S^2), \qquad (g_0,\dots,g_3)\mapsto (g_0\infty,g_1\infty,g_2\infty,g_3\infty).
\end{equation*} 
In \cite{DupontSah} it is shown that $\lambda$ has image in the Bloch group and that the following sequence, 
which is essentially due to Bloch and Wigner, is exact.
\begin{equation}\label{wigner}\xymatrix{{0}\ar[r]&{\Q/\Z}\ar[r] &{H_3(G)}\ar[r]^-\lambda &{\B(\C)}\ar[r]&{0}}\end{equation}
Using the isomorphism $\Q/\Z=\varinjlim {\Z/n\Z}=\varinjlim H_3(\Z/n\Z)$, 
the left map is the limit map induced by the maps $\Z/n\Z\to G$ given by sending $1$ to the matrix of a rotation by $2\pi/n$.
\subsection{Rogers' dilogarithm}
We here review a result in \cite{Dupont} relating the C--C--S class
$\hat P_1$ to a dilogarithm function via the Bloch group.

Rogers' dilogarithm is the following function defined on the open interval $(0,1)$:
\begin{equation}\label{rogdil} 
L(z) =-\frac{1}{2}\Log(z)\Log \Big(\frac{1}{1-z}\Big)+\Li(z)-\frac{\pi^2}{6}\end{equation}
Here $\Li(z)=-\smash{\int_0^z\frac{\Log(1-t)}{t}\,dt}$ is the classical
dilogarithm function. As in \cite{ParrySah} we have subtracted
$\unfrac{\pi^2}{6}$ from the original Rogers' dilogarithm in order to
make it satisfy \eqref{funceq}.
$L$ is real analytic and satisfies the functional equations
\begin{gather} L(x)+L(1-x)=-\frac{\pi^2}{6}\\
\label{funceq}L(x)-L(y)+L\Big(\frac{y}{x}\Big)-
L\Big(\frac{1-x^{-1}}{1-y^{-1}}\Big)+L\Big(\frac{1-x}{1-y}\Big)=0, \qquad y<x.\end{gather}
%for $0<y<x<1$.
We can extend $L$ (discontinuously) to $\R$ by setting
\begin{equation*} L(1)=0,\quad L(0)=-\frac{\pi^2}{6}\quad\text{and}\quad L(x)=\begin{cases}-L(1/x) \text{ for } x>1\\
       -L(\unpfrac{x}{x-1}) \text{ for } x<0\end{cases}\end{equation*}
and define a map $L\co  C_3(\SL(2,\R))\to \R$ by  
\begin{equation}\label{dupontchain}(g_0,\dots,g_3)\to L([g_0\infty:\dots:g_3\infty]).\end{equation} 
This is clearly well-defined (recall that cross-ratios are defined to be zero when there are equals)
since all cross-ratios are real.
Also, a few calculations using the functional equations show that the map takes boundaries
to multiples of $\pi^2/6$, that is, it is a $3$--cocycle modulo
$\pi^2/6$. The theorem below can be found in \cite{Dupont}.  The minus sign there is due to
the differing definition of the cross-ratio.

\begin{thm}\label{CCS} $\frac{1}{4\pi^2}L$ equals the Cheeger--Chern--Simons class $\hat P_1$ modulo $1/24$.
\end{thm}

Since the restriction of the second Chern polynomial to the Lie algebra of $\SL(2,\R)$ is minus the 
Pontrjagin polynomial, it follows from \fullref{godformel} that we have a commutative diagram:
\begin{equation}\label{PC}\cxymatrix{{{H_3(\SL(2,\R))}\ar[r]^-{-\hat P_1}\ar[d]&{\R/\Z}\ar[d]\\
           {H_3(\SL(2,\C))}\ar[r]^-{\hat C_2}&{\C/\Z}}}\end{equation}
By \fullref{CCS},  $\hat P_1$ is (modulo $1/24$) just a dilogarithm via the Bloch group.
We wish to find a similar expression for $\hat C_2$ by extending $L$ to $H_3(\SL(2,\C))$.
%We wish to show that $L$ (properly extended) is also a class via Neumann's extended Bloch group 
This is partially solved in \cite{Dupont} by studying a homomorphism $c\co  \B(\C)\to \C/\Q$, known as the Bloch regulator, 
and showing that the composition below is $2\hat C_2$.
\begin{equation*}\xymatrix{{H_3(\SL(2,\C))}\ar[r]^-\lambda& {\B(\C)}\ar[r]^c& {\C/\Q}}\end{equation*}
We shall improve this by showing that there is a commutative diagram
\begin{equation*}\cxymatrix{@C=50pt{{H_3(\SL(2,\C))}\ar[r]^-{\widehat\lambda}\ar[d]&{\widehat \B(\C)}\ar[r]^{-\frac{1}{2\pi^2}\widehat L}\ar[d]&{\C/\Z}\ar[d]\\
           {H_3(\SL(2,\C))/(\Q/\Z)}\ar[r]^-\lambda&{\B(\C)}\ar[r]^c&{\C/\Q}}}\end{equation*}
so that the top composition is $2\hat C_2$.
Here $\,\!\hatcB(\C)$ is Neumann's extended Bloch group (see \cite{Neumann} or \fullref{extb}).
In other words, $\hat C_2$ is a dilogarithm via the extended Bloch group exactly as
$\hat P_1$ is a dilogarithm via the Bloch group.

\section{The extended Bloch group}\label{extb}
In this section we review Neumann's definition of the extended Bloch
group. As mentioned in the introduction the reader should keep in mind that our extended Bloch group is what Neumann calls the more extended Bloch group. 

We shall use the conventions that the argument $\Arg z$ 
of a complex number always denotes the main argument  ($-\pi<\Arg z\leq \pi$) and the logarithm 
$\Log z$ always denotes the logarithm having $\Arg z$ as imaginary part.

The idea is to construct a Riemann surface $\,\!\widehat\C$ covering $\Cremove$ and then construct the extended pre-Bloch group 
$\widehat\Pre(\C)$ as in \fullref{predefn}, with an appropriate
lift of the five term relation. 

Let $\widehat \C$ denote the universal abelian cover of $\Cremove$. 
There is a nice way of representing points in $\widehat \C$.
Let $\C_{\cut}$ denote $\Cremove$ cut open along each of the intervals $(-\infty,0)$ and $(1,\infty)$ so 
that each real number $r$ outside of $[0,1]$ occurs twice in $\C_{\cut}$. We shall denote 
these two occurrences of $r$ by $r+0i$ and $r-0i$ respectively.
It is now easy to see that $\widehat\C$ is isomorphic to the surface obtained from $\C_{\cut}\times 2\Z\times 2\Z$ 
by the following identifications:
\begin{align*}(x+0i,2p,2q)&\sim(x-0i,2p+2,2q) \text{ for } x\in(-\infty,0)\\
(x+0i,2p,2q)&\sim(x-0i,2p,2q+2) \text{ for } x\in(1,\infty).\end{align*}
This means that points in $\widehat \C$ are of the form $(z,p,q)$ with $z\in \Cremove$ and $p,q$ even integers. Note that $\widehat\C$ can be regarded as the Riemann surface for the function
\[\Cremove\to \C^2,\qquad z\mapsto\Big(\Log z,\Log \Big(\frac {1}{1-z}\Big)\Big).\]
We shall show below that $L$ can be extended holomorphically to be defined on $\,\!\widehat \C$, and then
we shall simply define the extended five term relation to be the smallest possible extension of the relation \eqref{funceq}.

Consider the set
\[\FT:=\Big\{\Big(x,y,\frac{y}{x},\frac{1-x^{-1}}{1-y^{-1}},\frac{1-x}{1-y}\Big)\Big\}\subset(\Cremove)^5\]
of five-tuples involved in the five term relation.  Also let
\[\FT_0=\big\{(x_0,\dots,x_4)\in \FT\mid 0<x_1<x_0<1\big\}\]
be the set of five-tuples involved in the functional equation
\eqref{funceq}.  Define the set
$\widehat{\FT}\subset\widehat\C\times\dots\times\widehat\C$ to be the
component of the preimage of $\FT$ that contains all points
$\big((x_0;0,0),\dots,(x_4;0,0)\big)$ with $(x_0,\dots,x_4)\in\FT_0$.
\begin{remark}
This notation conflicts with the notation in \cite{Neumann}. Our $\widehat{\FT}$ is what Neumann calls $\widehat{\FT}_{00}$ in \cite[Section 8]{Neumann}. This is shown in the appendix. 
\end{remark}
\begin{defn}\label{ebgdefn}
The \emph{extended pre-Bloch group\/} $\widehat\Pre(\C)$ is the 
abelian group generated by symbols $[z;p,q]$, with $(z;p,q)\in \,\!\widehat\C$, subject to the relation
\[\sum_{i=0}^4(-1)^i[x_i;p_i,q_i]=0\textnormal{  for  }((x_0;p_0,q_0),\dots,(x_4;p_4,q_4))\in\,\!\widehat{\FT}.\]
This relation is called the \emph{extended five term relation\/}.
\end{defn}

\begin{defn}
The \emph{extended Bloch group\/} $\widehat \B(\C)$ is the kernel of the
homomorphism (which is well-defined by \cite[Lemma 2.3]{Neumann})
\[\widehat\nu\co  \widehat\Pre(\C)\to \C\wedge \C\] defined on generators by
$[z;p,q]\mapsto (\Log z+p\pi i)\wedge (-\Log(1-z)+q\pi i)$.
\end{defn}

We now extend $L$ to $\widehat \C$.
First note that the expression in \eqref{rogdil} is well-defined for all $z\in\Cremove$,
and that $L$ extended this way is holomorphic except at real points outside the interval \hbox{between $0$ and $1$.}
\[\widehat L(z;p,q)=L(z)+\frac{\pi i}{2}\Big(q\Log(z)-p\Log\Big(\frac{1}{1-z}\Big)\Big).
\leqno{\hbox{Define}}
\]

\begin{remark}
Neumann calls this map $R$ (probably for Rogers), but in fact Rogers originally called his dilogarithm $L$.
Also, the name $\widehat L$ is more consistent with our convention that all extended groups and maps be labelled with a hat.
\end{remark}
\begin{prop}{(\rm Neumann \cite[Proposition 2.5]{Neumann})}\qua
$\smash{\frac{1}{2\pi^2}}\widehat L$ gives a well-defined and holomorphic map $\widehat\C\to\C/\Z$.
Also the extended five term relation is a functional equation so that
$\frac{1}{2\pi^2}\widehat L$ gives a homomorphism $\widehat\Pre(\C)\to \C/\Z$.
\end{prop}

\subsection{Geometry of the extended pre-Bloch group}\label{geometry}
We first recall some geometric properties of the cross-ratio.
Let $z_0,z_1,z_2,z_3$ be four distinct ordered points in $\C\cup\{\infty\}$. Regarding $\C\cup\{\infty\}$ as the boundary
of the standard compactification of hyperbolic $3$--space $\mathbb H^3$, 
the four points define a unique ideal hyperbolic simplex $[z_0,\dots,z_3]$ which is 
determined up to orientation preserving congruence by the cross-ratio
\begin{equation}\label{cross-ratio}
z=[z_0:z_1:z_2:z_3]:=\frac{(z_0-z_3)(z_1-z_2)}{(z_0-z_2)(z_1-z_3)}.
\end{equation}
Clearly $z\in\Cremove$ and since $[0:\infty:1:z]=z$, every $z\in \Cremove$ can be realized 
as the cross-ratio of an ideal hyperbolic simplex.
It is well known that $z$ is real if and only if the four points lie on a circle (that is circle or straight line) and in this case the 
simplex is called \emph{flat\/}. 

We orient $\mathbb H^3$ such that the cross-ratio of a nonflat simplex has positive imaginary part if and only if the orientation induced by the vertex ordering agrees with the orientation inherited from $\mathbb H^3$. There is a nice geometric interpretation of the argument of $z$.
If the imaginary part of $z$ is greater than or equal to zero then $\Arg z$ is the dihedral angle of the simplex corresponding
to the edge $[z_0z_1]$. Otherwise, that is if the orientation disagrees with the orientation of $\mathbb H^3$, 
it is minus the dihedral angle.

 It easily follows from \eqref{cross-ratio} that an even permutation of the $z_i$'s replaces $z$ by one of three so-called
\emph{cross-ratio parameters\/}:
 \[z,\quad z'=\frac{1}{1-z}\quad\text{and}\quad z''=1-\frac{1}{z}\]
In particular the dihedral angle corresponding to the edges $[z_1z_2]$ and $[z_1z_3]$ are $\Arg(z')$ and 
$\Arg(z'')$ respectively, (or their negatives if the vertex ordering does not agree with the orientation of $\mathbb H^3$).
Since a product of two disjoint transpositions clearly keeps the cross-ratio fixed, we see that the dihedral angles of 
opposite edges are the same.
Note that since $zz'z''=-1$ the sum of the dihedral angles is always $\pi$. 
This is not surprising since a horosphere at an ideal vertex
of a hyperbolic simplex intersects the simplex in a Euclidean triangle.

\begin{defn}A \emph{combinatorial flattening\/} of an ideal simplex with cross-ratio $z$ 
is a triple $(w_0,w_1,w_2)$ of complex numbers with $w_0+w_1+w_2=0$,   
where $w_0$ and $w_1$ are choices of logarithms of $z$ and $z'$.  
We call $w_0,w_1$ and $w_2$ \emph{log-parameters\/}.
\end{defn} 

Note that $w_2\pm\pi i$ is a choice of logarithm of $z''$. The set of combinatorial flattenings of ideal simplices 
is in bijective correspondence with $\widehat \C$ by the map $l$ given by
\begin{equation}\label{l}l(w_0,w_1,w_2) = \Big(z;\frac{w_0-\Log z}{\pi i},\frac{w_1-\Log(\frac{1}{1-z})}{\pi i}\Big)\end{equation}
where $z=e^{w_0}$. This means that the extended pre-Bloch group can be regarded as being generated
by combinatorial flattenings of ideal simplices, whereas the pre-Bloch group can be regarded as being generated 
by congruence classes of ideal simplices. Let us discuss the five term relation in this geometric setup.

Suppose $(w_0,w_1,w_2)$ is a combinatorial flattening of an ideal simplex $[z_0,\dots,z_3]$.
Then we can assign log-parameters to each edge in such a way that $w_0$ is assigned to the edge $[z_0z_1]$,
$w_1$ to the edge $[z_1z_2]$ and $w_2$ to the edge $[z_1z_3]$. The three remaining edges are assigned the same
log-parameter as their opposite edge. See \mbox{\fullref{figure1}}.
Let $z_0,\ldots,z_4$ be five distinct points in $\C\cup\{\infty\}$ and let $\Delta_i$ denote
the simplices $[z_0,\dots,\hat z_i,\dots,z_4]$.
Using \eqref{cross-ratio}, it is easy to see that the cross-ratios $x_i$ of $\Delta_i$ can be expressed in terms of 
$x:=z_0$ and $y:=z_1$ as follows:
\begin{align*}x_0&=[z_1:z_2:z_3:z_4]:=x\\
x_1&=[z_0:z_2:z_3:z_4]:=y\\
x_2&=[z_0:z_1:z_3:z_4]=\frac{y}{x}\\
x_3&=[z_0:z_1:z_2:z_4]=\frac{1-x^{-1}}{1-y^{-1}}\\
x_4&=[z_0:z_1:z_2:z_3]=\frac{1-x}{1-y}\end{align*}
\begin{figure}[ht!]\label{figure1}
        \begin{center}
                \includegraphics[scale=0.9]{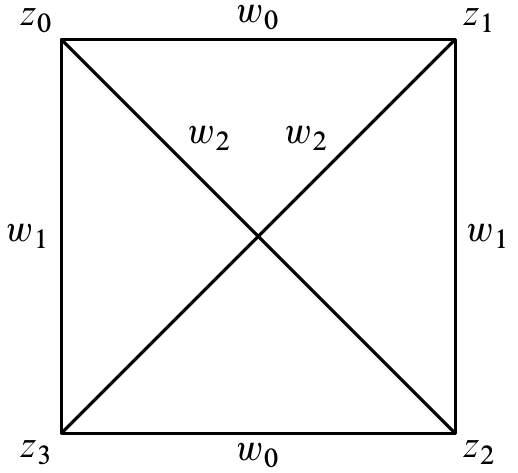}
        \caption{Assignment of log-parameters to edges of an ideal simplex}
        \end{center}
\end{figure}

Suppose $(w_0^i,w_1^i,w_2^i)$ are combinatorial flattenings of the simplices $\Delta_i$.
 Then every edge $[z_iz_j]$ belongs to exactly three of the $\Delta_i$'s 
and is therefore assigned three log-parameters.
\begin{defn}\label{combflat}  Let $(w_0^i,w_1^i,w_2^i)$ be combinatorial flattenings of the five simplices 
$\Delta_i=[z_0,\dots,\hat z_i,\dots,z_4]$.
The flattenings are said to satisfy the \emph{flattening condition\/} if for each edge
the signed sum of the three assigned log-parameters is zero (the sign is positive if and only if $i$ is even). 
\end{defn}
It follows directly from the definition that the flattening condition is equivalent to the following ten equations.
\begin{align*}
&[z_0z_1]:&w_0^2-w_0^3+w_0^4=0&& &[z_0z_2]:&-w_0^1-w_2^3+w_2^4=0\\
&[z_1z_2]:&w_0^0-w_1^3+w_1^4=0&& &[z_1z_3]:&w_2^0+w_1^2+w_2^4=0\\
&[z_2z_3]:&w_1^0-w_1^1+w_0^4=0&& &[z_2z_4]:&w_2^0-w_2^1-w_0^3=0\\
&[z_3z_4]:&w_0^0-w_0^1+w_0^2=0&& &[z_3z_0]:&-w_2^1+w_2^2+w_1^4=0\\
&[z_4z_0]:&-w_1^1+w_1^2-w_1^3=0&& &[z_4z_1]:&w_1^0-w_2^2-w_2^3=0
\end{align*}
Recall that combinatorial flattenings are in one to one correspondence with points in $\widehat \C$ via
the map $l$ in \eqref{l}.
\begin{thm}{(\rm Neumann \cite[Lemma 3.4]{Neumann})}\qua\label{neu} Flattenings $(w_0^i,w_1^i,w_2^i)$ satisfy the flattening condition if and only if
$\sum_{i=0}^4(-1)^i[l(w_0^i,w_1^i,w_2^i)]=0$ in $\,\!\widehat \Pre(\C)$. 
\end{thm}

This means that the flattening condition is equivalent to the extended five term relation.

\section[Mappings via configurations]{Mappings via configurations in ${\C^2}\backslash\{0\}$}\label{config}
In this section we explore the idea that the extra information needed to remove the $\Q/\Z$ indeterminacy 
in Dupont's formula for the C--C--S class $\hat C_2$ can be detected by
configurations in $\C^2\backslash\{0\}$  instead of $S^2$. Let $h$ denote the Hopf map $h\co  \C^2\backslash\{0\}\to S^2=\C\cup\{\infty\}$ given by 
\begin{equation*} (z,w)\mapsto z/w.\end{equation*}
We will show that for certain tuples $(v_0,\dots,v_3)$ of points in $\C^2\backslash\{0\}$, 
there is a natural choice of combinatorial flattening of the ideal simplex $[hv_0,\dots,hv_3]$.
This means that such a tuple gives an element in $\widehat \Pre(\C)$.
We also describe a way of associating such a tuple to a tuple of group elements in such a way that
we obtain a map \[\widehat\lambda\co  H_3(G)\to \widehat\Pre(\C).\]
Recall from \fullref{blochgroup} that there is a map $\sigma\co  C_3^{\neq}(S^2)_G\to \Pre(\C)$.
We saw that boundaries were mapped to zero and that the induced map 
\[\sigma \co H_3(C_*^{\neq}(S^2)_G)\to \Pre(\C)\]
is an isomorphism. 
We shall elaborate on this and construct a $G$--complex $C_*^{h\neq}(\C^2)$ and a map $\,\!\widehat \sigma\co  C_3^{h\neq}(\C^2)_G\to \widehat\Pre(\C)$ giving rise to a commutative diagram:
\begin{equation*}
\xymatrix{{H_3(C_*^{h\neq}(\C^2)_G)}\ar[r]^-{\widehat\sigma}\ar[d]^h&{\widehat\Pre(\C)}\ar[d]\\
            {H_3(C_*^{\neq}(S^2)_G)}\ar[r]^-\sigma&{\Pre(\C)}}
\end{equation*} 
We define the complex $C_*^{h\neq}(\C^2)$ as the subcomplex of $C_*(\C^2\backslash\{0\})$ consisting of tuples 
mapping to different elements in $S^2$ by the Hopf map $h$. The $G$--module structure is given by the natural $G$--action on $\C^2\backslash\{0\}$, and since this action is $h$--equivariant, $h$ induces a $G$--map $C_*^{h\neq}(\C^2)\to C_*^{\neq}(S^2)$ and hence a map 
\[h\co  H_3(C_*^{h\neq}(\C^2)_G)\to H_3(C_*^{\neq}(S^2)_G).\] 

\subsection{Mapping to the extended pre-Bloch group}
We now assign to each $4$--tuple $(v_0,v_1,v_2,v_3)\in C_3^{h\neq}(\C^2)$ a combinatorial flattening
of the ideal simplex $[hv_0,hv_1,hv_2,hv_3]$ 
in such a way that the combinatorial flattenings assigned to tuples 
$(v_0,\dots,\hat{v_i},\dots v_4)$ satisfy the flattening condition.
This will give us a map \[\widehat\sigma\co  H_3(C_*^{h\neq}(\C^2)_G)\to \widehat \Pre(\C).\] 
The key step is to observe that the cross-ratio parameters $z$ and $\frac{1}{1-z}$ of a
simplex $[hv_0,hv_1,hv_2,hv_3]$ can be 
expressed in terms of determinants
\begin{equation*}z:=[hv_0:hv_1:hv_2:hv_3]=\frac{\left(\frac{v^1_0}{v^2_0}-\frac{v^1_3}{v^2_3}\right)}{\left(\frac{v^1_0}{v^2_0}-\frac{v^1_2}{v^2_2}\right)}
\frac{\left(\frac{v^1_1}{v^2_1}-\frac{v^1_2}{v^2_2}\right)}{\left(\frac{v^1_1}{v^2_1}-\frac{v^1_3}{v^2_3}\right)}=
\frac{\det(v_0,v_3)\det(v_1,v_2)}{\det(v_0,v_2)\det(v_1,v_3)}\end{equation*}
where the upper indices refer to first or second coordinate in $\C^2$.
Similarly, \begin{equation*}\frac{1}{1-z}=[hv_0:hv_2:hv_0:hv_3]=\frac{\left(\frac{v^1_1}{v^2_1}-\frac{v^1_3}{v^2_3}\right)}{\left(\frac{v^1_1}{v^2_1}-\frac{v^1_0}{v^2_0}\right)}
\frac{\left(\frac{v^1_2}{v^2_2}-\frac{v^1_0}{v^2_0}\right)}{\left(\frac{v^1_2}{v^2_2}-\frac{v^1_3}{v^2_3}\right)}=
\frac{\det(v_1,v_3)\det(v_0,v_2)}{\det(v_0,v_1)\det(v_2,v_3)}.\end{equation*} 
Since obviously $hv_i\neq hv_j$ if and only if $\det (v_i,v_j)\neq 0$, all these determinants are nonzero.
This suggests that we can assign a flattening to $(v_0,v_1,v_2,v_3)$ by setting  
\begin{align*} 
w_0=&\Log \det(v_0,v_3)+\Log \det(v_1,v_2)-\Log \det(v_0,v_2)-\Log \det(v_1,v_3)\\
w_1=&\Log \det(v_0,v_2)+\Log \det(v_1,v_3)-\Log \det(v_0,v_1)-\Log \det(v_2,v_3)\\
w_2=&\Log \det(v_0,v_1)+\Log \det(v_2,v_3)-\Log \det(v_0,v_3)-\Log \det(v_1,v_2).\end{align*}
This defines a map $\widehat \sigma\co   C_3^{h\neq}(\C^2)\to \widehat\Pre(\C)$ by
\begin{equation}(v_0,v_1,v_2,v_3)\mapsto [l(w_0,w_1,w_2)].\end{equation}
Now suppose $(w_0^0,w_1^0,w_2^0),\dots ,(w_0^4,w_1^4,w_2^4)$ are flattenings defined as abo\-ve
of simplices $[hv_0,\dots,\widehat {hv_i},\dots ,hv_4]$.
We must check that these flattenings satisfy the flattening condition.
This is equivalent to checking that all the ten equations listed below 
\fullref{combflat} are satisfied. 
We check the first of these and leave the others to the reader.
Using the notation $(v,w):=\Log \det (v,w)$ we have \eject
$\phantom{99}$\vspace{-5mm}
\begin{align*}w_0^2&=(v_0,v_4)+(v_1,v_3)-(v_0,v_3)-(v_1,v_4)\\
w_0^3&=(v_0,v_4)+(v_1,v_2)-(v_0,v_2)-(v_1,v_4)\\
w_0^4&=(v_0,v_3)+(v_1,v_2)-(v_0,v_2)-(v_1,v_3)
\end{align*}
from which it follows that the equation $w_0^2-w_0^3+w_0^4=0$ is satisfied.

Having verified all the ten equations, it now follows from \fullref{neu} that $\widehat \sigma$ sends boundaries to zero. Since $\widehat \sigma$ obviously factors through $C_3^{h\neq}(\C^2)_G$, we obtain a map 
%$\widehat \sigma(v_0,\dots,v_3)=\widehat \sigma(g v_0,\dots,g v_3)$, 
$\,\!\widehat \sigma\co  H_3(C_*^{h\neq}(\C^2)_G)\to \,\!\widehat\Pre(\C)$.

It is clear that the diagram below is commutative.
\begin{equation*}
\xymatrix{{H_3(C_*^{h\neq}(\C^2)_G)}\ar[r]^-{\widehat\sigma}\ar[d]^h&{\widehat\Pre(\C)}\ar[d]\\
          {H_3(C_*^{\neq}(S^2)_G)}\ar[r]^-{\sigma} &{\Pre(\C)}}
\end{equation*}

\begin{prop}\label{bhat} The image of $\widehat \sigma\co  H_3(C_*^{h\neq}(\C^2)_G)\to \widehat\Pre(\C)$ is in $\widehat \B(\C)$.
\end{prop}
\begin{proof}
Define a map $\mu\co C_2^{h\neq}(\C^2)_G\to \C\wedge\C$ by
\[(v_0,v_1,v_2) \mapsto (v_0,v_1)\wedge(v_0,v_2)-(v_0,v_1)\wedge (v_1,v_2)+(v_0,v_2)\wedge (v_1,v_2)\]
where we still use the notation $(v,w):=\Log \det (v,w)$.
A straightforward calculation shows that the diagram below is commutative. 
\[\xymatrix{{C_3^{h\neq}(\C^2)_G}\ar[r]^-{\widehat\sigma}\ar[d]^\partial&{\widehat\Pre(\C)}\ar[d]^{\widehat\nu}\\
           {C_2^{h\neq}(\C^2)_G}\ar[r]^-{\mu}&{\C\wedge \C}}\]
This means that cycles are mapped to $\widehat \B(\C)$ as desired.
\end{proof}
    \subsection[The map from the third homology group of G]{The map from $H_3(G)$}
In this section we shall construct a map $\widehat\lambda$ from $H_3(G)$ to $\widehat\Pre(\C)$ via the group $H_3(C_*^{h\neq}(\C^2)_G)$.
To define this map explicitly on the chain level we need to restrict to a subcomplex of $C_*(G)$.
\begin{defn} A chain in $C_*(G)$ is called \emph{good\/} if all its tuples satisfy $g_i\neq \pm g_j$ 
and \emph{$v$--good\/} ($v\in \C^2$) if all its tuples satisfy $\det(g_i v,g_j v)\neq 0$. 
The $G$--complexes of good and $v$--good chains are denoted $C_*^{\text{good}}(G)$ and $C_*^v(G)$ respectively.
\end{defn}
By \fullref{subcomp}, $C_*^{\text{good}}(G)$ and $C_*^v(G)$
are both acyclic so $C_*^{\text{good}}(G)_G$ and $C_*^v(G)_G$  both
calculate the homology of $G$. From now on we will identify $H_3(G)$ with $H_3(C_*^{\text{good}}(G)_G)$.
Consider the $G$--maps
\begin{align*} \Psi_v\co  &C_n(G)\to C_n(\C^2),&(g_0,\dots,g_n)&\mapsto (g_0v,\dots ,g_nv)\\
 \text{conj}_g\co  &C_n(G)\to C_n(G),&(g_0,\dots,g_n)&\mapsto (gg_0g^{-1},\dots,gg_ng^{-1}).
\end{align*}
Note that if $\sigma$ is in $C_*^v(G)_G$ then $\text{conj}_g(\sigma)$ is in $C_*^{gv}(G)_G$ and we have
\begin{equation}\label{conjugation}\Psi_{gv}(\text{conj}_g(\sigma))=\Psi_v(\sigma).\end{equation} 
It is clear that $\Psi_v$ takes $v$--good chains to $C_n^{h\neq}(\C^2)$.

The following is simple.
\begin{lemma}\label{qweer} Let $g_1\neq \pm g_2\in G$.
The subset \[\{v\in \C^2\mid \det(g_1 v,g_2 v)\neq 0\}\subset \C^2\] is open and dense.
\end{lemma}
For a good chain $\sigma$ belonging to either $C_*^{\text{good}}(G)$ or $C_*^{\text{good}}(G)_G$
consider the set \[S_\sigma=\{v\in \C^2\mid \sigma\text{ is $v$--good} \}.\] 
Since finite intersections of dense open subsets is dense open, it follows from 
\mbox{\fullref{qweer}} that
$S_\sigma$ is dense open. In other words, any good chain is also a $v$--good chain for almost all $v\in \C^2$. 
The following is a simple consequence of \eqref{conjugation} and the
well-known fact that conjugation induces the identity map on homology.
\begin{prop}Let $\sigma \in C_*^{\textnormal{good}}(G)_G$ be a cycle.
The homology class of $\Psi_v(\sigma)$ is independent of $v\in S_\sigma$. 
\end{prop}
We can now define a map $\Psi\co H_3(G)\to H_3(C_*^{h\neq}(\C^2)_G)$
by \begin{equation*}[\sigma]\mapsto
  [\Psi_v(\sigma)],\qquad v\in S_\sigma.\end{equation*}

\begin{prop}The diagram below is commutative.
\begin{equation*}\xymatrix{{H_3(G)}\ar[r]^-{\Psi}\ar[d]&{H_3(C_*^{h\neq}(\C^2)_G)}\ar[rd]^h&{}\\
  {H_3(G)/\Q/\Z}\ar[r]^-\lambda&{\Pre(\C)}&{H_3(C_*^{\neq}(S^2)_G)}\ar[l]_-\sigma}\end{equation*} 
\end{prop}
\begin{proof}
The map $\Psi$  obviously coincides with the map
\begin{equation*}
\xymatrix{{H_3(G)\cong H_3(C_*^v(G)_G)}
\ar[r]^-{\Psi_v}& {H_3(C_*^{h\neq}(\C^2)_G).}}
\end{equation*}
The proposition follows from this with $v=\binom{1}{0}$,
since $h\binom{1}{0}=\infty$.
\end{proof}
We can now define $\widehat\lambda$ as the composition
\begin{equation*}
\xymatrix{{H_3(G)}\ar[r]^-{\Psi}&{H_3(C_*^{h\neq}(\C^2)_G)}\ar[r]^-{\widehat\sigma}&{\hatcB(\C).}}
\end{equation*}

\begin{remark} The second author has shown that
  $H_3(C_*^{h\neq}(\C^2)_G)$ is canonically isomorphic to $H_3(G,P)$,
  where $P$ is the parabolic subgroup of upper triangular matrices with $1$ on the diagonal. Under this isomorphism $\Psi$ 
  corresponds to the natural map $H_3(G)\to H_3(G,P)$. 
  This result makes \fullref{hovedsats} more
  directly applicable to hyperbolic manifolds, since a hyperbolic
  manifold with cusps has a natural fundamental class in
  $H_3(G,P)$. More on this will appear elsewhere.
\end{remark}

\section{Relation with the Cheeger--Chern--Simons class}\label{relccc}
In this section we relate the maps constructed above to the Cheeger--Chern--Simons class $\hat C_2$.
Our goal is to prove the following theorem.
\begin{thm}\label{hovedsats} $-\frac{1}{2\pi^2}\widehat L\circ\widehat\lambda=2\hat C_2$.
\end{thm}
\begin{remark}
The reader who wishes to compare this result with Neumann's may notice that the factors of $2$ seem to be missing in \cite[Theorem 12.1]{Neumann}.
This is because Neumann uses a different normalisation of the C--C--S class.
It is well known that the natural map $H_3(\SL(2,\C))\to H_3(\PSL(2,\C))$ is surjective with kernel $\Z/4\Z$, %This follows from the Hochshild--Serre spectral sequence for the exact sequence $0\to \Z/2\Z\to \SL(2,\C)\to \PSL(2,\C)\to 0$. 
so from \fullref{godformel} we get a commutative diagram:
\[\xymatrix{0\ar[r]& {\Z/4\Z}\ar@2{-}[d]\ar[r]&H_3(\SL(2,\C))\ar[d]^{\hat C_2}\ar[r]&H_3(\PSL(2,\C))\ar[d]^{\hat C_2}\ar[r]&0\\0\ar[r]&{\Z/4\Z}\ar[r]&{\C/\Z}\ar[r]&{\C/\frac{1}{4}\Z}\ar[r]&0}\]
In Neumann's normalisation $\C/\frac{1}{4}\Z$ is identified with $\C/\pi^2\Z$ via the map sending $x$ to $(2\pi i)^2 x$ and it follows that our result agrees with that of Neumann. The reason for using Neumann's normalisation is that $\hat C_2$ evaluated on the fundamental class of a complete hyperbolic manifold with finite volume is $i(\Vol + i\text{CS})$ as mentioned in the introduction. The invariant $\Vol +i\text{CS}$ is often regarded as a natural complexification of volume, so from this point of view this normalisation seems more natural. We have, however, chosen to keep the original normalisation to make references to earlier papers easier. 
\end{remark}
Let $H_3(G)_{\pm}$ denote the subgroups
$\{x\in H_3(G)\mid\tau x=\pm x\}$
where $\tau$ is the involution induced by complex conjugation. We
shall refer to these subgroups as the real and the imaginary parts of
$H_3(G)$.

The following is simple.

\begin{prop}
$\frac{1}{2\pi^2}\widehat L\circ\widehat\lambda$ is equivariant under complex conjugation.
\end{prop} From \fullref{godformel}, $\hat C_2$ is also equivariant under
conjugation, and since $H_3(G)$ is divisible by a result in
\cite{DupontSah}, it is enough
to study the real and imaginary parts separately. 
\subsection{The imaginary part}
It is well known that the oriented volume of an ideal simplex with cross-ratio $z$ is given by
\[\Vol(z)=\Arg(1-z)\Log\vert z\vert-\Imag\int_0^1\frac{\Log(1-tz)}{t}\,dt.\]
For a proof of this see \cite[page 172]{DupontSah}. 
\begin{remark}
As mentioned earlier Dupont and Sah \cite{DupontSah} use a different cross-ratio convention, but the formula for the oriented volume remains unchanged, since they orient  $\mathbb H^3$ according to their cross-ratio (the orientation of a simplex with cross-ratio $z$ is positive if and only if $\Imag(z)>0$).\end{remark}
The five term relation \eqref{fiveterm} is easily seen to be a functional equation for $\Vol$. This means that $\Vol$ is well-defined on the pre-Bloch group and therefore also on the extended pre-Bloch group.

\begin{thm}{(\rm Dupont \cite[Proposition 3.1]{Dupont})}\qua\label{dupontim} $\Imag\hat C_2=-\frac{1}{4\pi^2}\Vol\circ \lambda$.
\end{thm}
\begin{prop}\label{neuim} The restriction of $\Imag\widehat L\co \widehat\Pre(\C)\to\R$ to $\hatcB(\C)$ equals $\Vol$. 
\end{prop}
\begin{proof} 
Let $\tau=\sum(-1)^{\epsilon_i}[z_i;p_i,q_i]\in\hatcB(\C)$.
Since \begin{align*}
\Imag \widehat L([z;p,q])=&\frac{1}{2}\big(\Arg(z)\Log\vert 1-z\vert+\Log\vert z\vert \Arg(1-z)\big)\\
&-\Imag\int_0^1\frac{\Log(1-tz)}{t}\,dt+\frac{\pi}{2}p\Log\vert 1-z\vert+
  \frac{\pi}{2}q\Log\vert z\vert\\
\tag*{\hbox{we have}}\Vol(z)-\Imag &\widehat L([z;p,q])=\frac{1}{2}\big(\Log\vert z\vert \Arg(1-z)-\Arg(z)\Log\vert 1-z\vert\big)
\\&\qquad\qquad\qquad-\frac{\pi}{2}p\Log\vert 1-z\vert-\frac{\pi}{2}q\Log\vert z\vert.\end{align*}
Let $\phi$ denote the composition 
\[\C\wedge \C=\left(\R\wedge\R\right)\oplus \left(i\R\wedge i\R\right)\oplus\left(\R\otimes i\R\right) \to \R\otimes i\R\to i\R\] 
where the left map is projection
and the right map is multiplication.
A simple calculation shows that 
\begin{multline*}\phi\big(\widehat\nu([z;p,q])\big)=-i\Log\vert z\vert \Arg(1-z)+i\Arg(z)\Log\vert 1-z\vert\\
+p\pi i\Log\vert 1-z\vert+q\pi i\Log\vert z\vert=-2i\big(\Vol(z)-\Imag \widehat L([z;p,q])\big).\end{multline*}
Since $\widehat\nu(\tau)=0$, we have $\Vol(\tau)=\Imag\widehat L(\tau)$ as desired.
\end{proof}
\subsection{The real part} 
Let $G_\R=\SL(2,\R)$.
The key step is the following theorem of Dupont, Parry and Sah \cite{DupontParrySah,Sah}.
\begin{thm}
The inclusion $G_\R\to G$ induces an isomorphism
\[H_3(G_\R)\cong H_3(G)_+.\]
\end{thm}
This means that it is enough to study real cycles. The idea is that
every homology class in $H_3(G_\R)$ has a representative such that the
image of $\widehat L\circ \widehat \lambda$ is the same as the image
of the cocycle $L$ from \eqref{dupontchain}. 

In the following the reader should bear in mind the relationship
between the homogenous and the inhomogenous representations of cycles given by the
equations \eqref{inhom} and \eqref{hom}. %We will from now on use letters $a_i$ to denote group elements involved in a homogenous representation, and letters $g_i$ to denote elements in the corresponding inhomogenous representation. We have $g_i=a_{i-1}^{-1}a_i$.

\begin{defn} An element $\big(\begin{smallmatrix}a&b\\c&d\end{smallmatrix}\big)\in G_\R$ is called \emph{positive\/} if $c$ is positive and \emph{nonzero\/} if $c$ is nonzero. A chain in $B_n(G_\R)$ is called \emph{positive\/} if all its group elements are positive.\end{defn}

If $(g_1,g_2,g_3)$ is a triple of positive elements that are so small
(close to the identity) that also $g_1g_2$, $g_2g_3$ and $g_1g_2g_3$
are positive, we have
\begin{itemize}
\item
  $(v_0,v_1,v_2,v_3):=\big(\binom{1}{0},g_1\binom{1}{0},g_1g_2\binom{1}{0},g_1g_2g_3\binom{1}{0}\big)$ is in $C_3^{h\neq}(\C^2)$.
\item $\det(v_i,v_j)>0$ for $i<j$.
\item $\infty>g_1\infty>g_1g_2\infty>g_1g_2g_3\infty$.
\end{itemize} 
The third property ensures that the
cross-ratio $z$ of the associated flat ideal simplex is strictly between
$0$ and $1$, and the second property ensures that the log-parameters $w_0,w_1,w_2$ satisfy that
$l(w_0,w_1,w_2)=(z;0,0)$.  This means that if $\alpha$ is an inhomogenous
representation of a class in $H_3(G_\R)$ with all group elements
sufficiently small and positive then 
\begin{equation}\label{small} \frac{1}{2\pi^2}\widehat L\circ\widehat\lambda(\alpha)=\frac{1}{2\pi^2}L(\alpha).
\end{equation}
As we shall see below, every homology class in $H_3(G_\R)$ has such a representative.

The following is essentially just an application of barycentric subdivision, and we refer to \cite[Proposition 2.8]{Dupont} for a proof.
\begin{lemma}\label{dupontlemma} Let $H$ be a contractible Lie group and $U$ a neighborhood of the identity. Every cycle in $B_*(H)$ is homologous to a cycle consisting of elements in $U$.
\end{lemma}
Let $\widetilde{G_\R}$ be the universal covering group of $G_\R$. 
Parry and Sah \cite{ParrySah} analyse the Hochshild--Serre spectral sequence for the exact sequence \[0\to \Z\to \widetilde{G_\R}\to G_\R\to 0\]
and obtain:
\begin{prop}\label{parrysah} $H_3(\widetilde {G_\R})\to H_3(G_\R)$ is surjective.
\end{prop}

Since $G_\R$ is homotopy equivalent to a circle, $\widetilde{G_\R}$ is
contractible, and by \mbox{\fullref{dupontlemma}} and \fullref{parrysah}, every homology class in $H_3(G_\R)$ has an
inhomogenous representative with all group elements arbitrarily
small. \\
Let $U$ be an open neighborhood of the identity in $G_\R$ satisfying
that any product of up to three positive elements is positive.

We now show that every sufficiently small cycle in $B_3(G_\R)$ is
homologous to a positive cycle with all elements in $U$. This implies that every homology class has a representative satisfying \eqref{small}.\\
%Note also, that by equation \eqref{hom}, an inhomogenous representation of a homology class is positive if and only if all terms $(g_0,\dots,g_n)$ in the corresponding homogenous representation satisfy that $g_{i-1}^{-1}g_{i}$ is positive for $i=1,\dots,n$.
 
Define an ordering of elements in $G_\R$ by
\[g_1<g_2 \iff g_1^{-1}g_2 \textnormal{ is positive.}\]
This ordering is neither total nor transitive, but as we shall see,
this can be fixed.
The following is simple.
\begin{lemma}\label{un}
For every natural number $n$ there exists an open subset $U_n$ of $U$ satisfying that $g\in U_n$ if and only if
  $g^{-1}\in U_n$ and that any product of up to $n$ positive elements in
  $U_n$ is a positive element in $U$. 
\end{lemma}
Fix neighborhoods $U_n$ as above. We may assume that
$U_n\subset U_{n-1}$ and that the product of any two elements in $U_n$
is in $U_{n-1}$.
\begin{defn} Let $k\leq n$. A $k$--chain in $B_k(G_\R)$ is called a 
$U_n$--$k$--chain if it is nonzero and all its group elements lie in $U_n$. 
A $k$--chain in $C_k(G_\R)$ is called
a $U_n$--$k$--chain if it maps to a $U_n$--$k$--chain in $B_k(G_\R)$. The set of $U_n$--$k$--chains in $C_k(G_\R)$ is denoted $C_k(G_\R)_{U_n}$. 
\end{defn}
\begin{prop}\label{bubblesort}
Let $g_0,\dots ,g_n\in G_\R$ satisfy that all elements $g_{i-1}^{-1}g_i$ are in $U_n$ and nonzero.
There exists a \emph{unique\/} permutation $\sigma\in S_{n+1}$ such that 
$g_{\sigma(0)}<\dots<g_{\sigma(n)}$.
\end{prop}
\begin{proof}
The assumption on the $g_i$'s implies that the restriction of the ordering to $\{g_0,\dots,g_n\}$ is transitive, and since $\big(\begin{smallmatrix}a&b\\c&d\end{smallmatrix}\big)^{\smash{-1}}=\big(\begin{smallmatrix}d&-b\\-c&a\end{smallmatrix}\big)$, we have either $g_i<g_j$ or $g_i>g_j$.
This means that we can use the bubble sort algorithm to produce the desired permutation.
\end{proof} 
We thereby obtain $G_{\R}$--maps
\[\Psi_k\co  C_k(G_{\R})_{U_n}\to C_k(G_{\R}).\]
Note that the image of $\Psi_k$ consists of chains whose images
in $B_k(G_\R)$ consist entirely of positive elements in $U$. Note also
that the boundary map takes $C_k(G_\R)_{U_n}$ to $C_{k-1}(G_\R)_{U_{n-1}}$.
\begin{prop}\label{homcycle}
Let $\tau\in C_k(G_{\R})_{U_n}$, $k\leq n$, represent a cycle in
$B_k(G_{\R})$. Then $\Psi_k(\tau)$ and $\tau$ represent homologous
cycles in $B_k(G_{\R})$.
\end{prop}
\begin{proof} By the uniqueness in \fullref{bubblesort}, the
  maps $\Psi_k$ give rise to a chain map in dimensions up to $n$. By a
  standard argument, there exist $G_{\R}$--maps $S_k\co  C_k(G_\R)_{U_n}\to C_{k+1}(G_\R)$, $k=0,\dots ,n$, such that
\[\partial S_k+S_{k-1}\partial = \Psi_k-\text{id}.\]
This proves the assertion.
\end{proof}

\begin{proof}[Proof of \fullref{hovedsats}] 
By \fullref{homcycle} and \fullref{dupontlemma} we have that
every homology class in $H_3(G_{\R})$ has an inhomogenous
representative satisfying  \eqref{small}. Recall from diagram \eqref{PC} that the restriction of $\hat C_2$ to $H_3(G_{\R})$ equals $-\hat P_1$. 
By\break equation \eqref{small}, \fullref{neuim}, \fullref{dupontim} and \fullref{CCS}, we have that\break
$-\smash{\frac{1}{2\pi^2}}\widehat L\circ\widehat\lambda-2\hat C_2$ has image in $\smash{\frac{1}{12}}\Z\big/\Z=\Z/12\Z$.
As mentioned earlier $H_3(G)$ is divisible, which means that it has no nontrivial
finite quotient. Thus  $2\smash{\hat C_2}=-\smash{\frac{1}{2\pi^2}}\smash{\widehat L}\circ\smash{\widehat\lambda}$ as required.\end{proof}
The rest of this section is devoted to a proof of \fullref{jaja}
below, but in order to prove this theorem, we need to recall some
properties of $\hat C_2$ and the relationship between $\widehat \B(\C)$ and $\B(\C)$.

Recall from  \eqref{wigner} that $\Q/\Z$ can be regarded as a subgroup
of $H_3(\SL(2,\C))$. It is known that the restriction of $\hat C_2$ to this subgroup is
just the inclusion $\iota$ of $\Q/\Z$ in $\C/\Z$. In other words, we have a commutative diagram:
{\disablesubscriptcorrection\begin{equation}\label{csid}
\cxymatrix{{{\Q/\Z}\ar@{^{(}->}[r]\ar@{^{(}->}[rd]_\iota&{H_3(\SL(2,\C))}\ar[d]^{\hat C_2}\\{}&{\C/\Z}}}
\end{equation}}%
For a proof of this see \cite[Theorem 10.2, remarks on page 60]{Scissors}. 

Neumann shows in \cite[Corollary 8.3]{Neumann} that $\widehat \B(\C)$ and $\B(\C)$ are related by an exact sequence
\begin{equation}\label{nex}
\xymatrix{{0}\ar[r]&{\Q/\Z}\ar[r]^{\widehat \chi}&{\widehat \B(\C)}\ar[r]&{\B(\C)}\ar[r]&{0}}
\end{equation}
where $\widehat\chi$ is the map given by
\begin{equation*}
\widehat\chi(z)=[e^{2\pi iz};0,2]-[e^{2\pi iz};0,0].
\end{equation*}
\begin{thm}\label{jaja} The map $\widehat \lambda\co  H_3(\SL(2,\C))\to \widehat \B(\C)$ is surjective with kernel $\Z/2\Z$.
\end{thm}
\begin{proof}
Suppose $\widehat \lambda(\alpha)=0$. By composing with the map to $\B(\C)$, we see from \eqref{wigner} that $\alpha$ is in $\Q/\Z$. By \eqref{csid} and \fullref{hovedsats}, we have
\begin{equation*}
0=-\frac{1}{2\pi^2}\widehat L\circ\widehat\lambda(\alpha)=2\hat C_2(\alpha)=2\alpha.
\end{equation*}
Hence, $\alpha$ is either zero or the unique element in $\Q/\Z$ of order $2$.

Let $\alpha\in \widehat \B(\C)$. 
A simple calculation shows that we have
\begin{equation*} -\frac{1}{2\pi^2}\widehat L\circ\widehat\chi=\iota,
\end{equation*}
and using \eqref{csid} we get a commutative diagram:
{\disablesubscriptcorrection\begin{equation}\label{andet}
\cxymatrix{{{\Q/\Z}\ar[r]\ar[rd]_{2\widehat \chi}&{H_3(\SL(2,\C))}\ar[d]^{\widehat \lambda}\\&{\widehat \B(\C)}}}
\end{equation}}%
Let $\pi$ denote the natural map $\hatcB(\C)\to \B(\C)$, and let $x$ be an element in $H_3(\SL(2,\C))$ satisfying $\pi(\alpha)=\lambda(x)$. By \eqref{nex}, there exists $z$ in $\Q/\Z$ such that 
\begin{equation*}
\widehat \lambda(x)-\alpha=\widehat \chi(z),\end{equation*}
and by \eqref{andet}, we have $\widehat\lambda(x-\tfrac{1}{2}z)=\alpha$. %Thus, $\hat\lambda$ is surjective.
\end{proof}

\section*{Appendix}\label{pendix}

We conclude by proving that our definition of $\hatcB(\C)$ is
equivalent to Neumann's definition of the more extended Bloch group
$\mathcal E\B(\C)$. This actually follows directly from the brief remark in parentheses on the bottom of page 417 in \cite{Neumann},
but we give the details to save the reader some trouble. Recall the definition of $\FT$ from \fullref{extb}.  Neumann defines
\begin{equation*}
  \FT^+:=\{(x_0,\dots,x_4)\in \FT\mid \Imag x_i>0\}
\end{equation*}
and defines $\widehat{\FT}_{00}$ to be the component of
the preimage of $\FT$ that contains all points
\begin{multline}\label{pq}
\big((x_0;p_0,q_0),(x_1;p_1,q_1),(x_2;p_1-p_0,q_2),\\(x_3;p_1-p_0+q_1-q_0,q_2-q_1),(x_4;q_1-q_0,q_2-q_1-p_0)\big)
\end{multline} 
with $(x_0,\dots,x_4)\in\FT^+$ and the $p_i$'s and $q_i$'s even integers. 
He then defines the more extended Bloch group $\mathcal E\B(\C)$, as in \fullref{ebgdefn}, to be  the abelian group generated by symbols $[z;p,q]$, subject to the relation
\[\sum_{i=0}^4(-1)^i[x_i;p_i,q_i]=0\textnormal{  for  }\big((x_0;p_0,q_0),\dots,(x_4;p_4,q_4)\big)\in\widehat{\FT}_{\text{00}}.\]
\begin{prop}
$\hatcB(\C)=\mathcal E\B(\C)$.
\end{prop}
This follows immediately from the following lemma.
\begin{lemma}\label{applemma}
$\widehat\FT_{\textnormal{00}}=\widehat\FT$.
\end{lemma}
\begin{proof}
Let $(x_0,\dots,x_4)$ be a fixed point in $\FT^+$ and let \[P=((x_0;0,0),\dots,(x_4;0,0))\in \widehat \FT_{\text{00}}.\] Consider 
the curve in $\widehat \FT_{\text{00}}$ starting in $P$ obtained by keeping $x_1$ fixed and letting $x_0$ move along a closed curve in $\C-\{0,1,x_1\}$.
By a simple analysis of the five term relation, we can examine exactly how the values of the $p_i$'s and $q_i$'s change when $x_0$ moves around. This is indicated in Figure $2$. 
\begin{figure}[ht!]
        \begin{center}
                \includegraphics[scale=0.9]{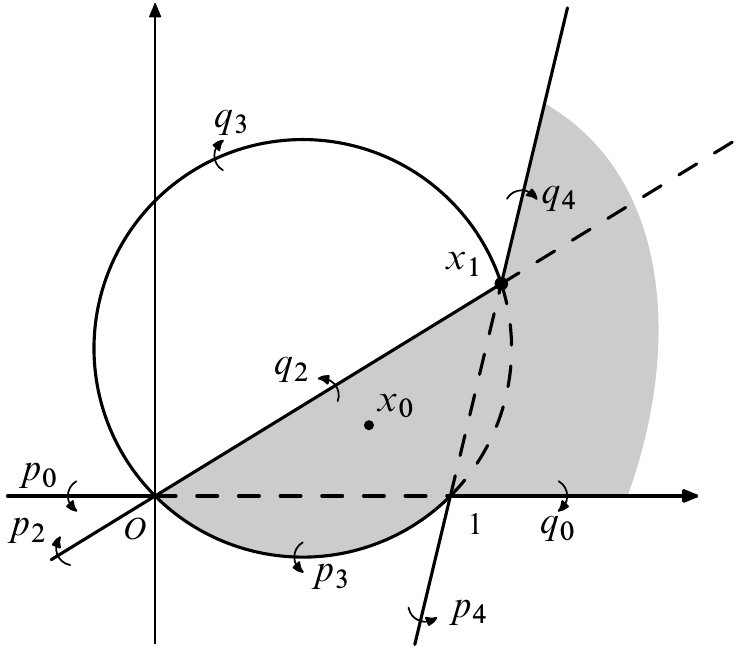}
                \caption{The lines in the figure are the cuts of the
                  function sending $z$ to $(\Log(z),\Log(\frac{1}{1-z}))$ in the
                  $x_i$--plane, $i=0,2,3,4$, when $y=x_1$ is fixed.
 The relevant values of $p_i$ and $q_i$
                  increase by $2$ whenever $x_0$ crosses the relevant
                  line in the direction indicated by the arrows.}
        \end{center}
\end{figure} 
We see that if $x_0$ traverses a closed curve going $p_0$ times counterclockwise around the origin, followed by $q_0$ times clockwise around $1$, followed by $r$ times clockwise around $x_1$, then the curve in $\widehat \FT_{\text{00}}$ ends in
\begin{multline*} \big((x_0;2p_0,2q_0),(x_1;0,0),(x_2;-2p_0,2p_0+2r),\\(x_3;-2p_0-2q_0,2p_0+2r),(x_4;-2q_0,2r)\big).\end{multline*}
If we start in this point and then follow the curve in $\widehat\FT_{\text{00}}$ obtained by keeping $x_0$ fixed and letting $x_1$ traverse a curve going $p_1$ times counterclockwise around the origin followed by $q_1$ times clockwise around $1$, a similar study shows that we end up at the point    
\begin{multline*} Q=\big((x_0;2p_0,2q_0),(x_1;2p_1,2q_1),(x_2;-2p_0+2p_1,2p_0+2r),\\(x_3;-2p_0-2q_0+2p_1+2q_1,2p_0-2q_1+2r),(x_4;-2q_0+2q_1,-2q_1+2r)\big).\end{multline*}
By letting $q_2=2r+2p_0$ we see that $Q$ is of the form \eqref{pq}. Since we can connect $P$ to a point in $\widehat \FT$ by first sliding $x_0$ down to the interval $(0,1)$ and then doing the same with $x_1$, the lemma follows.
\end{proof}

\bibliographystyle{gtart}
\bibliography{link}

\end{document}